\documentclass[a4paper,11pt]{article}

\usepackage{moreverb}

\usepackage[dvips,colorlinks,bookmarksopen,bookmarksnumbered,citecolor=red,urlcolor=red]{hyperref}

\newcommand\BibTeX{{\rmfamily B\kern-.05em \textsc{i\kern-.025em b}\kern-.08em
T\kern-.1667em\lower.7ex\hbox{E}\kern-.125emX}}

\usepackage[T1]{fontenc}
\usepackage[utf8]{inputenc}
\usepackage{mathtools}
\usepackage{amsfonts}
\usepackage{enumitem}
\usepackage{mathrsfs}
\usepackage{amsthm}
\usepackage{layaureo}

\DeclarePairedDelimiter{\abs}{\lvert}{\rvert}
\DeclarePairedDelimiter{\norma}{\lVert}{\rVert}
\DeclareMathOperator{\supp}{supp}

\newcommand{\N}{\mathbb{N}}
\newcommand{\Z}{\mathbb{Z}}

\newcommand{\diff}{\mathrm{d}}

\newtheorem{theorem}{Theorem}
\newtheorem{lemma}{Lemma}

\newtheorem{definition}{Definition}

\newcommand{\cC}{\,{\mathcal C}\,}

\newcommand{\cE}{\,{\mathcal E}\,}
\newcommand{\cEE}{\,{\mathcal{EE}}\,}

\newcommand{\Econ}{\stackrel{\cE}{\rightarrow}}
\newcommand{\EEcon}{\stackrel{\cEE}{\rightarrow}}
\newcommand{\Ccon}{\stackrel{\cC}{\rightarrow}}

\begin{document}
%

\title{Boundary homogenization for a triharmonic intermediate problem
}

\author{Jos\'{e} M. Arrieta, Francesco Ferraresso and Pier Domenico Lamberti
}


\maketitle

\begin{abstract}
We consider the triharmonic  operator subject to homogeneous boundary conditions of intermediate type on a bounded domain of the N-dimensional Euclidean space. We study its spectral behaviour when the boundary of the domain undergoes a perturbation of oscillatory type.  We identify the appropriate limit problems which depend on whether the strength of the oscillation is above or below a critical threshold. We analyse in detail the critical case which provides a typical homogenization problem leading to a strange boundary term in the limit problem.
\end{abstract}

%


\vspace{-6pt}

\section{Introduction}

Given a sufficiently regular bounded domain $\Omega$ in ${\mathbb{R}}^N$ with $N\geq 2$ and $f\in L^2(\Omega)$, we consider the following
boundary value problem
\begin{equation}
\label{classical}
\left\{
\begin{array}{ll}
-\Delta^3u+u=f,& \ \ {\rm in}\ \Omega,\\
u=0,& \ \ {\rm on}\ \partial\Omega,\\
\frac{\partial u}{\partial \nu}=0,& \ \ {\rm on}\ \partial\Omega,\\
\frac{\partial^3 u}{\partial \nu^3}=0,& \ \ {\rm on}\ \partial\Omega,
\end{array}
\right.
\end{equation}
where $\nu $ denotes the unit outer normal to $\partial \Omega$. The variational weak formulation of problem (\ref{classical})
reads
\begin{equation}\label{weak}
\int_{\Omega} \Bigl( D^3u:D^3\varphi +u\varphi \Bigr)\,dx =\int_{\Omega} f\varphi\,dx,\ \ \ \forall \varphi \in W^{3,2}(\Omega)\cap W^{2,2}_0(\Omega),
\end{equation}
in the unknown $u \in  W^{3,2}(\Omega)\cap W^{2,2}_0(\Omega)$.  Here $D^3 v=\{ \frac{\partial ^3v}{\partial x_i\partial x_j\partial _k} \}_{i,j,k=1,2,3}$ denotes the set of all derivatives of order three of a function $v$ and $D^3u:D^3\varphi =\sum_{i,j,k=1}^3 \frac{\partial ^3u}{\partial x_i\partial x_j\partial _k}\frac{\partial ^3\varphi}{\partial x_i\partial x_j\partial _k} $is the usual Frobenius product. Moreover,
$ W^{k,2}(\Omega)$ denotes  the Sobolev space of functions in $L^2(\Omega )$ with weak derivatives in $L^2(\Omega)$ up to order $k$ endowed with its standard norm, and $ W^{k,2}_0(\Omega)$ the closure  in $W^{k,2}(\Omega)$ of the space $C^{\infty}_c(\Omega ) $  of smooth functions with compact support in $\Omega$. We note that the first two boundary conditions in (\ref{classical}) are encoded in the condition $u\in W^{2,2}_0(\Omega)$, while the third boundary condition in (\ref{classical}) is the natural boundary condition arising from integrating by parts the left-hand side in \eqref{weak}, see e.g., \cite[Chp. 1, Prop. 2.4]{Nec}.

Recall that in the classical Dirichlet problem for the triharmonic operator, the boundary condition
$\frac{\partial^3 u}{\partial \nu^3}=0$ in (\ref{classical}) is replaced by the condition $\frac{\partial^2 u}{\partial \nu^2}=0$
and the corresponding weak problem can be formulated  exactly as in (\ref{weak}) with $W^{3,2}(\Omega)\cap W^{2,2}_0(\Omega)$ replaced by $W^{3,2}_0(\Omega)$. We note that using the energy space $W^{3,2}(\Omega )$ in (\ref{weak}) rather than $W^{3,2}(\Omega)\cap W^{2,2}_0(\Omega)$ would lead to a Neumann-type boundary value problem in the same spirit of standard Neumann  problems for the Laplace or  the biharmonic operator. Thus, since the energy space  $W^{3,2}(\Omega)\cap W^{2,2}_0(\Omega)$ used in (\ref{weak}) satisfies the inclusions $W^{3,2}_0(\Omega)\subset W^{3,2}(\Omega)\cap W^{2,2}_0(\Omega)\subset W^{3,2}(\Omega)$,    we refer to problem (\ref{classical}) as  an intermediate problem.  For an introduction to the theory of polyharmonic operators we refer to the extensive monograph \cite{GazzGS}.

By standard operator theory, problems (\ref{classical}) and (\ref{weak}) can be recast in the form
\begin{equation}
H_{\Omega}u=f,
\end{equation}
where $H_{\Omega}$ is a positive self-adjoint operator densely defined in $L^2(\Omega)$ such that the domain ${\rm Dom} (H_{\Omega}^{1/2})$ of its square root $H_{\Omega}^{1/2}$ is $W^{3,2}(\Omega)\cap W^{2,2}_0(\Omega)$ and such that $<H_{\Omega}^{1/2}u,H_{\Omega}^{1/2} \varphi >_{L^2(\Omega )}= \int_{\Omega}D^3u:D^3\varphi +u\varphi dx$ for all $u,\varphi \in W^{3,2}(\Omega)\cap W^{2,2}_0(\Omega)$. Moreover, $u\in {\rm Dom} (H_{\Omega})$ if and only if $u\in {\rm Dom} (H_{\Omega}^{1/2})$ and there exists $f\in L^2(\Omega)$ such that  $<H_{\Omega}^{1/2}u,H_{\Omega}^{1/2} \varphi >_{L^2(\Omega )}=<f, \varphi >_{L^2(\Omega )} $ for all $\varphi \in {\rm Dom} (H_{\Omega}^{1/2})$, in which case $H_{\Omega}u=f$. If $\Omega$ is sufficiently regular (a Lipschitz continuous boundary is enough) then the resolvent $H^{-1}_{\Omega}$ is compact, hence the spectrum of $H_{\Omega}$ is discrete.  Formally, the operator $H_{\Omega}$ can be identified with the classical operator
$-\Delta^3 + \mathbb{I}$ subject to the boundary conditions in (\ref{classical})

In this paper, we continue the analysis addressed in  \cite{Comprendou, ArrLamb} for the case of the biharmonic operator, and we study the  compact convergence of  the resolvent operators $H^{-1}_{\Omega_{\epsilon}}$ defined on suitable families of domains $\{\Omega_{\epsilon }\}_{\epsilon >0}$ approaching  a fixed domain $\Omega$ as $\epsilon \to 0$.
As in \cite{Comprendou, ArrLamb}, in order to simplify the setting, we suppose that $\Omega = W \times (-1,0)$ where $W$ is a sufficiently regular bounded domain of ${\mathbb{R}}^{N-1}$,   $\Omega_{\epsilon}= \{(\bar{x},x_N)\in {\mathbb{R}}^N : \bar{x}\in W,\, -1 < x_N < g_\epsilon(\bar{x}) \}$ where  $g_\epsilon(\bar{x})= \epsilon^\alpha b(\bar{x}/\epsilon)$ for all $\bar{x}\in W$, and $b$  is a $Y$-periodic smooth function with $Y=\left(-1/2,1/2\right)^{N-1}$.

Compact convergence is a standard notion in functional analysis and is equivalent to the convergence in operator norm in the case of self-adjoint operators defined on a fixed Hilbert space.  In our case, the underlying  Hilbert space is the space  $L^2(\Omega_{\epsilon})$   which depends on $\epsilon$. This leads to a number of technical difficulties which can be overcome by using  the notion of ${\mathcal{E}}$-compact convergence where ${\mathcal{E}}$ denotes an operator which allows to pass from the reference Hilbert space $L^2(\Omega )$ to the other Hilbert spaces $L^2(\Omega_{\epsilon})$. In our setting, ${\mathcal{E}}$  is just the extension-by-zero operator which can be thought as an operator from $L^2(\Omega)$ to $L^2(\Omega_{\epsilon})$ defined by ${\mathcal{E}}u=u_0|{\Omega_{\epsilon}}$ for all $u\in L^2(\Omega)$, where $u_0$ is the function defined by $u$ in $\Omega$ and zero outside $\Omega$. For the convenience of the reader we recall the following definition (see e.g., \cite{arbr} and the references therein).

\begin{definition}\label{e-convergence}
i) We say that $v_\epsilon\in L^2(\Omega_\epsilon)$ $\mathcal{E}$-converges to $v\in L^2(\Omega)$ if $\|v_\epsilon-\mathcal{E}v\|_{L^2(\Omega_\epsilon)}\to 0$ as $\epsilon\to 0$. We write this as
$v_\epsilon\Econ v$.
\par\noindent ii) The family of bounded linear operators $B_\epsilon\in \mathcal{L}(L^2(\Omega_\epsilon))$ $\mathcal{EE}$- converges to $B\in \mathcal{L}(L^2(\Omega))$ if $B_\epsilon v_\epsilon\Econ Bv$ whenever $v_\epsilon\Econ v$. We write this as
$B_\epsilon\EEcon B$.
\par\noindent iii) The family of bounded linear and compact  operators $B_\epsilon\in \mathcal{L}(L^2(\Omega_\epsilon))$ $\mathcal{E}$-compact converges to $B\in \mathcal{L}(L^2(\Omega))$ if $B_\epsilon\EEcon B$ and for any family of functions $v_\epsilon\in L^2(\Omega_\epsilon)$ with $\|v_\epsilon\|_{L^2(\Omega_\epsilon)}\leq 1$ there exists a subsequence, denoted by $v_\epsilon$ again, and a function $w\in L^2(\Omega)$ such that $B_\epsilon v_\epsilon \Econ w$. We write
$B_\epsilon\Ccon B$.
\end{definition}

  We note that the  ${\mathcal{E}}$-compact convergence of the resolvent operators $H^{-1}_{\Omega_{\epsilon }}$ implies not only the convergence of the solutions $u_{\epsilon}$ of the Poisson problems  $H_{\Omega_{\epsilon}}u_{\epsilon }=f$ but  also the convergence of the eigenvalues and eigenfunctions of the operators $H_{\Omega_{\epsilon }}$.

Our main result is the following theorem  which can be considered as the triharmonic analogue of \cite[Theorem 7.3]{ArrLamb}  concerning a problem somewhat close to the so-called Babuska Paradox for the biharmonic operator (see also the important contributions to this subject in \cite{Maz-Naz, Maz-Naz-Pla}). Here and in the sequel  the part of the boundary of $\Omega$ given by $W\times \{0\}$ is denoted by $W$.

\begin{theorem}\label{mainthm}
With the notation above, the following statements hold true.
\begin{itemize}
\item[(i)] \textup{[Spectral stability]} If $\alpha > 3/2$, then $H^{-1}_{\Omega_\epsilon } \overset{\mathcal{C}}{\rightarrow} H^{-1}_{\Omega }$.
\item[(ii)] \textup{[Strange term]} If $\alpha = 3/2$, then $H^{-1}_{\Omega_\epsilon,I} \overset{\mathcal{C}}{\rightarrow} \hat{H}^{-1}_{\Omega}$, where $\hat{H}_\Omega$ is the operator $-\Delta^3 + \mathbb{I}$ with intermediate boundary conditions on $\partial \Omega \setminus W$ and the following boundary conditions on $W$: $u=\frac{\partial u}{\partial x_n}=\frac{\partial^3 u}{\partial x_n^3} - K \frac{\partial^2 u}{\partial x_n^2} = 0$, where the factor $K$ is given by

   \begin{equation} \label{kappa}
    K =  \int_{Y\times (-\infty,0)} |D^3 V|^2\, d{y}=  - \int_Y \Bigg( \Delta\Bigg(\frac{\partial^2 V}{\partial y_N^2}\Bigg) + 2 \Delta_{N-1}\Bigg(\frac{\partial^2 V}{\partial y_N^2}\Bigg) \Bigg) b(\bar{y}) d{\bar{y}},
    \end{equation}
  $\Delta_{N-1}$ denotes the Laplacian in the first N-1 variables,  and $V$ is a function, $Y$-periodic in the variables $\bar{y}$, satisfying the following microscopic problem
    \[
    \begin{cases}
    \Delta^3 V = 0, &\textup{in $Y \times (-\infty, 0)$},\\
    V(\bar{y}, 0) = 0, &\textup{on $Y$},\\
    \frac{\partial V}{\partial y_N}(\bar{y}, 0) = b(\bar{y}), &\textup{on $Y$},\\
    \frac{\partial^3 V}{\partial y_N^3}(\bar{y}, 0) = 0, &\textup{on $Y$}.
    \end{cases}
    \]
    \item[(iii)] \textup{[Degeneration]} If $0< \alpha < 3/2$ and $b$ is non-constant, then $H^{-1}_{\Omega_\epsilon } \overset{\mathcal{C}}{\rightarrow} H^{-1}_{\Omega, D}$, where $H_{\Omega, D}$ is the operator $-\Delta^3 + \mathbb{I}$ with intermediate boundary conditions on $\partial \Omega \setminus W$ and  classical Dirichlet boundary conditions on $W$, namely $u=\frac{\partial u}{\partial x_n}=\frac{\partial^2 u}{\partial x_n^2}=0$ on $W$.
\end{itemize}
\end{theorem}

We note that the analysis of the cases $\alpha \le 3/2 $ is in spirit of the paper \cite{casado10} which is devoted to the Navier-Stokes system.  For  recent results concerning domain perturbation problems for higher order operators we refer to \cite{BuosoLamb, BuosoLamb2, BP, BurLamb, bula2012}.

\section{Proof of Theorem~\ref{mainthm}}

In this section we provide the proof the Theorem~\ref{mainthm}.
For simplicity, we shall always assume that $b\geq 0$, hence $\Omega \subset \Omega_{\epsilon }$ for all $\epsilon >0$.\\
Let $f_\epsilon\in L^2(\Omega_\epsilon)$ and $f\in L^2(\Omega)$ be such that $\norma{f_\epsilon}_{L^2(\Omega_\epsilon)}$ is uniformly bounded and $f_\epsilon \rightharpoonup f$ in $L^2(\Omega )$ as $\epsilon \to 0$. Let $v_\epsilon \in  W^{3,2}(\Omega_\epsilon) \cap W^{2,2}_0(\Omega_\epsilon)$ be such that
\begin{equation}
\label{eq: Poisson prblm}
H_{\Omega_\epsilon, I} v_\epsilon = f_\epsilon\, ,
\end{equation}
for all $\epsilon >0$ small enough.   We plan to pass to the limit in (\ref{eq: Poisson prblm}) as $\epsilon \to 0$ and prove that the limit problem is as in  Theorem~\ref{mainthm}.
Clearly, $\norma{v_\epsilon}_{W^{3,2}(\Omega_\epsilon)} \leq M$ for all $\epsilon > 0$ sufficiently small, hence, possibly passing to a subsequence, there exists $v \in W^{3,2}(\Omega)\cap W^{2,2}_0(\Omega)$ such that $v_\epsilon \rightharpoonup v$ in $W^{3,2}(\Omega )$ and $v_\epsilon \rightarrow v$ in $L^2(\Omega)$.
In order to use the weak formulation of  problem (\ref{eq: Poisson prblm}), we need to define a   suitable test function in $\Omega_{\epsilon}$ starting from a given test function in $\Omega$. Following the approach in \cite{ArrLamb}, this is done by means of an appropriate pullback operator. Namely, we consider a diffeomorphism $\Phi_\epsilon$ from $\Omega_{\epsilon}$ to $\Omega$ defined by $\Phi_{\epsilon}(\bar x, x_N)=(\bar x, x_N-h_\epsilon (\bar{x}, x_N))$ for all $(\bar x, x_N)\in \Omega_{\epsilon}$ where
\[
h_\epsilon (\bar{x}, x_N) =
\begin{cases}
0, &\textup{if $-1 \leq x_N \leq -\epsilon$},\\
g_\epsilon(\bar{x})\Big(\frac{x_N+\epsilon}{g_\epsilon(\bar{x})+\epsilon}\Big)^4, &\textup{if $-\epsilon \leq x_N \leq g_\epsilon(\bar{x})$}.
\end{cases}
\]
The map $\Phi_\epsilon$ is a diffeomorphism of class $C^3$, even though the highest order derivatives may not be uniformly bounded as $\epsilon \to 0$. Note in particular that  there exists a constant
$c > 0$ independent of $\epsilon$ such that for all $\epsilon > 0$ small enough we have
\[
\abs{h_\epsilon} \leq c \epsilon^\alpha, \quad \left \lvert \frac{\partial h_\epsilon}{\partial x_i} \right \rvert \leq c \epsilon^{\alpha - 1}, \quad \left \lvert \frac{\partial^2 h_\epsilon}{\partial x_i \partial x_j}  \right \rvert \leq c \epsilon^{\alpha - 2}, \quad \left \lvert\frac{\partial^3 h_\epsilon}{\partial x_i \partial x_j \partial x_k} \right \rvert \leq c \epsilon^{\alpha - 3}.
\]
Then we consider the pullback operator $T_\epsilon$ from $L^2(\Omega)$ to $L^2(\Omega_\epsilon)$ defined  by
$T_\epsilon \varphi = \varphi \circ \Phi_\epsilon$,
for all $\varphi\in L^2(\Omega)$.

Let $\varphi \in W^{3,2}(\Omega) \cap W^{2,2}_0(\Omega)$ be  fixed. Since $T_\epsilon \varphi \in W^{3,2}(\Omega_{\epsilon}) \cap W^{2,2}_0(\Omega_{\epsilon}) $, by \eqref{eq: Poisson prblm} we get
\begin{equation}
\label{eq: Poisson 1}
\int_{\Omega_\epsilon} D^3v_\epsilon : D^3 T_\epsilon \varphi \, \diff{x} + \int_{\Omega_\epsilon}v_\epsilon T_\epsilon \varphi\, \diff{x} = \int_{\Omega_\epsilon} f_\epsilon T_\epsilon \varphi\, \diff{x},
\end{equation}
and passing to the limit as $\epsilon \to 0$ we have that
\begin{equation}
\label{eq: limit easy terms}
\int_{\Omega_\epsilon} v_\epsilon T_\epsilon \varphi\, \diff{x} \to \int_{\Omega} v \varphi \,\diff{x}, \quad \int_{\Omega_\epsilon} f_\epsilon T_\epsilon \varphi\, \diff{x} \to \int_{\Omega} f \varphi \,\diff{x}.
\end{equation}
Now we consider the first integral in the left-hand side of \eqref{eq: Poisson 1} and set $K_\epsilon = W \times (-1, -\epsilon)$. By splitting the integral in three terms corresponding to $\Omega_\epsilon \setminus \Omega$, $\Omega \setminus K_\epsilon$ and $K_\epsilon$ and by arguing as in \cite[Section 8.3]{ArrLamb} one can show that
\begin{equation}
\label{eq: limit less easy terms}
\int_{K_\epsilon} D^3v_\epsilon : D^3 T_\epsilon \varphi \, \diff{x} \to \int_\Omega D^3 v : D^3 \varphi \, \diff{x}, \quad \int_{\Omega_\epsilon \setminus \Omega} D^3 v_\epsilon : D^3 T_\epsilon \varphi \, \diff{x} \to 0,
\end{equation}
as $\epsilon \to 0$. Hence, it remains to analyse the behaviour of the term $\int_{\Omega\setminus K_\epsilon} D^3v_\epsilon : D^3 T_\epsilon \varphi \, \diff{x} $. We distinguish now the three cases.

{\bf Case $\alpha >3/2$.}  In this case,  one can  prove that $\int_{\Omega\setminus K_\epsilon} D^3v_\epsilon : D^3 T_\epsilon \varphi \, \diff{x}\to  0$. Thus, by combining the previous limit relations, we get
$
\int_{\Omega } D^3v : D^3 \varphi \, \diff{x} + \int_{\Omega}v \varphi\, \diff{x} = \int_{\Omega } f \varphi\, \diff{x},
$
which proves statement (i).

{\bf Case $\alpha =3/2$.}  In this case,  the problem is more complicated and the proof of statement (ii) will follow from Theorems~\ref{thm: macroscopic limit} and \ref{last}.  The proofs of such theorems are based on  the unfolding method (see e.g., \cite{ciodamgri}). We recall now a few notions from homogenization theory.  For any $k \in \Z^{N-1}$ and $\epsilon >0$ we consider the small cell $C^k_\epsilon = \epsilon(k+Y)$, where as above $Y = \Bigl(-\frac{1}{2}, \frac{1}{2} \Bigr)^{N-1}$. Let $I_{W,\epsilon} = \{k \in \Z^{N-1} : C^k_\epsilon \subset W \}$ and $\widehat{W}_\epsilon = \bigcup_{k \in I_{W,\epsilon} } C^k_\epsilon $.
We set $Q_\epsilon = \widehat{W_\epsilon} \times (-\epsilon,0)$
and we split again the remaining integral  in two summands, namely
\begin{equation}\label{eq: integralsplit}
\int_{\Omega \setminus K_\epsilon} D^3v_\epsilon : D^3 T_\epsilon \varphi \, \diff{x} = \int_{\Omega \setminus (K_\epsilon \cup Q_\epsilon)} D^3v_\epsilon : D^3 T_\epsilon \varphi \, \diff{x} + \int_{Q_\epsilon} D^3v_\epsilon : D^3 T_\epsilon \varphi \, \diff{x}.
\end{equation}
Arguing as in \cite[Section 8.3]{ArrLamb} we get that
$
\int_{\Omega \setminus (K_\epsilon \cup Q_\epsilon)} D^3v_\epsilon : D^3 T_\epsilon \varphi \, \diff{x} \to 0
$
as $\epsilon \to 0$. Thus, it remains to study the limiting behaviour of the last summand in the right hand-side of \eqref{eq: integralsplit} and this is done by unfolding it. We recall the following

\begin{definition}
Let $u$ be a measurable real-valued function defined in $\Omega$. For any $\epsilon > 0$ sufficiently small the unfolding $\hat{u}$ of $u$ is the function defined on $\widehat{W}_\epsilon \times Y \times (-1/\epsilon, 0)$ by
$
\hat{u}(\bar{x}, \bar{y}, y_N) = u\Big( \epsilon\bigl[\frac{\bar{x}}{\epsilon}\bigr] + \epsilon y, \epsilon y_N \Big),
$
for almost all $(\bar{x}, \bar{y}, y_N)) \in \widehat{W}_\epsilon \times Y \times (-1/\epsilon, 0)$, where $\bigl[\frac{\bar{x}}{\epsilon}\bigr]$ denotes the integer part of the vector $\bar{x} \epsilon^{-1}$ with respect to $Y$, i.e., $\bar{x} \epsilon^{-1} = k$ if and only if $\bar{x} \in C^k_\epsilon$.
\end{definition}

The unfolding operator allows to `unfold' integrals by means of the well-known exact integration formula which in our case can be written as
\begin{equation}\label{exact}\int_{\widehat W_{\epsilon}\times (a, 0  ) }u(x)dx=\epsilon\int_{\widehat W_{\epsilon}\times Y\times (a/ \epsilon, 0  )  }\hat u (\bar x, y)d\bar xdy,
\end{equation}
for any $a\in [-1,0[$. This formula will be essential in computing the limit of   $\int_{Q_\epsilon} D^3v_\epsilon : D^3 T_\epsilon \varphi \, \diff{x}$ as $\epsilon\to 0$. Before doing this, we need two technical lemmas.
By $w^{3,2}_{Per_Y}(Y \times (-\infty, 0))$ we denote the space of functions $u$ in $W^{3,2}_{loc}({\mathbb{R}}^{N-1}\times (-\infty ,0))$ which are $Y$-periodic in the first $N-1$ variables and such that
$\| D^{\eta}u\|_{L^2(Y\times (-\infty , 0))}<\infty $ for all $|\eta|=3$.

\begin{lemma}
\label{lemma: unfolding convergence}
Let $v_\epsilon \in W^{3,2}(\Omega)$ with $\norma{v_\epsilon}_{W^{3,2}(\Omega)} < M$, for all $\epsilon > 0$. Let $V_\epsilon(\bar{x}, y) = \hat{v_\epsilon}(\bar{x}, y)-P(\hat{v_\epsilon}(\bar{x}, y)) $ for all $(\bar{x}, y) \in \widehat{W_\epsilon}\times Y\times (-1/\epsilon, 0)$,  where the operator $P$ is defined by
\begin{multline*}
P(w(\bar{x}, y))
 = \int_Y\biggl( w   (\bar{x}, \bar{y}, 0)  -   \sum_{\abs{\eta}=2}\int_Y D^\eta_{\bar y} w  (\bar{x}, \bar{y}, 0)\, \diff{\bar{y}}\, \frac{\bar y^\eta}{\eta!} \biggr)  \, \diff{\bar{y}}\\
 + \int_Y \nabla_y w (\bar{x}, \bar{y}, 0)\, \diff{\bar{y}} \cdot y + \sum_{\abs{\eta}=2}\int_Y D^\eta_y w(\bar{x}, \bar{y}, 0)\, \diff{\bar{y}}\, \frac{y^\eta}{\eta!} \, .
\end{multline*}
Then there exists  $\hat{v}\in L^2(W, w^{3,2}_{\textup{Per}_Y}(Y\times (-\infty,0)))$ such that
\begin{enumerate}[label=(\alph*)]
\item $\frac{D_y^{\eta}V_\epsilon}{\epsilon^{5/2}} \rightharpoonup D_y^{\eta}\hat{v}$ in $L^2(W\times Y \times (d,0))$ as $\epsilon \to 0$, for all $d<0$ and $\eta \in \N^N$, $\abs{\eta} \leq 2$.
\item $\frac{D_y^{\eta}V_\epsilon}{\epsilon^{5/2}} =\frac{D_y^{\eta}\hat{v_\epsilon}}{\epsilon^{5/2}}\rightharpoonup D_y^{\eta}\hat{v}$ in $L^2(W\times Y \times (-\infty,0))$ as $\epsilon \to 0$, for all $\eta \in \N^N$, $\abs{\eta} = 3$,
\end{enumerate}
where it is understood that functions $V_\epsilon, D_y^{\eta}V_\epsilon$ are extended by zero in the whole of $W\times Y \times (-\infty,0)$ outside their natural domain of definition $\widehat{W_\epsilon}\times Y\times (-1/\epsilon, 0) $.
\end{lemma}
\begin{proof}
The proof is similar to the one in~\cite[Lemma 8.9]{ArrLamb}.  Using  formula (\ref{exact}), one can easily prove that
 $\left\| \epsilon^{-5/2}D^{\eta }_yV_{\epsilon}\right\|_{L^2(W\times Y\times (-\infty , 0))
} $ is uniformly bounded with respect to $\epsilon$ for all $|\eta|=3$.
Note that the operator $P$  is a projector on the space of polynomials   of the second degree in the variable $y$.  Thus,  a standard argument exploiting  a Poincar\'{e}-Wirtinger-type inequality implies the existence of a real-valued function $\hat v$ defined on $W\times Y\times (-\infty , 0)$ which admits weak derivatives up to the third order locally in the variable $y$,   such that statements (a) and (b) hold.
In order to prove the periodicity of $\hat v$ in $\bar y$, we can apply to $D^2 V_\epsilon$ an argument similar to the one contained in Lemma 4.3 in \cite{casado10}  to obtain that $\nabla_y \hat{v}$ is periodic. Then we find out that $\hat{v}$ is also periodic because
$\int_{Y}\nabla_y \hat{v}(\bar{x},\bar{y},0) \diff{\bar{y}} = 0$, being this true for
all the functions $V_\epsilon$. \end{proof}

\begin{lemma}
\label{lemma: convergence b}
For all $y \in Y \times (-1,0)$ and $i,j,k = 1,\dots,N$ the functions $\hat{h}_\epsilon(\bar{x},y)$, $\widehat{\frac{\partial h_\epsilon}{\partial x_i}}(\bar{x},y)$, $\widehat{\frac{\partial^2 h_\epsilon}{\partial x_i\partial x_j}}(\bar{x},y)$ and $\widehat{\frac{\partial^3 h_\epsilon}{\partial x_i\partial x_j \partial x_k}}(\bar{x},y)$ are independent of $\bar{x}$. Moreover, $\hat{h}_\epsilon(\bar{x},y) = O(\epsilon^{3/2})$, $\widehat{\frac{\partial h_\epsilon}{\partial x_i}}(\bar{x},y) = O(\epsilon^{1/2})$ as $\epsilon \to 0$,
\[
\epsilon^{1/2} \widehat{\frac{\partial^2 h_\epsilon}{\partial x_i\partial x_j}} (\bar{x},y) \to \frac{\partial^2 (b(\bar{y})(y_N + 1)^4)}{\partial y_i \partial y_j},
\]
as $\epsilon \to 0$, for all $i,j = 1,\dots,N$, uniformly in $y \in Y \times (-1,0)$, and
\[
\epsilon^{3/2} \widehat{\frac{\partial^3 h_\epsilon}{\partial x_i\partial x_j \partial x_k}} (\bar{x},y) \to \frac{\partial^3 (b(\bar{y})(y_N + 1)^4)}{\partial y_i \partial y_j \partial y_k},
\]
as $\epsilon \to 0$, for all $i,j,k = 1,\dots,N$, uniformly in $y \in Y \times (-1,0)$.
\end{lemma}
\begin{proof}
It is a matter of easy but lengthy calculations, which can be carried out as in \cite[Lemma 8.27]{ArrLamb}.
\end{proof}

Now we are ready to prove the following

\begin{theorem}
\label{thm: macroscopic limit}
Let $M$ be a positive real number. Let $f_\epsilon \in L^2(\Omega_\epsilon)$, $\norma{f_\epsilon}_{L^2(\Omega_\epsilon)} < M$ for all $\epsilon > 0$ and $f\in L^2(\Omega)$ be such that $f_\epsilon \rightharpoonup f$ in $L^2(\Omega)$. Let $v_\epsilon\in W^{3,2}(\Omega_\epsilon) \cap W^{2,2}_0(\Omega_\epsilon)$ be the solutions to $H_{\Omega_\epsilon} v_\epsilon = f_\epsilon$. Then, up to a subsequence, there exists $v\in W^{3,2}(\Omega) \cap W^{2,2}_0(\Omega)$ and $\hat{v} \in L^2(W,w^{3,2}_{Per_Y}(Y\times(-\infty,0)))$ such that $v_\epsilon\rightharpoonup v$ in $W^{3,2}(\Omega)$, $v_\epsilon \to v$ in $L^2(\Omega)$, statements $(a)$ and $(b)$ in Lemma~\ref{lemma: unfolding convergence} hold, and such that for each $\varphi \in W^{3,2}(\Omega) \cap W^{2,2}_0(\Omega)$ the following holds
\footnotesize{
\begin{eqnarray}\label{weakmacro}\lefteqn{  \int_{\Omega} D^3 v : D^3\varphi + u\varphi \,\diff{x}
- 3 \int_W \int_{Y \times (-1,0)} D^2_y\Big(\frac{\partial \hat{v}}{\partial y_N} \Big) : D^2_y (b(\bar{y})(1 + y_N)^4) \,\diff{y}\, \frac{\partial^2\varphi}{\partial x_N^2}(\bar{x}, 0) \diff{\bar{x}}}\nonumber   \\ & &
\qquad\qquad
- \int_W \int_{Y \times (-1,0)} y_N (D^3_y(\hat{v}) : D^3 (b(\bar{y})(1+y_N)^4) \,\diff{y}\, \frac{\partial^2\varphi}{\partial x_N^2}(\bar{x}, 0) \diff{\bar{x}}
= \int_{\Omega} f \varphi\, \diff{x}.
\end{eqnarray}
}
\end{theorem}

\begin{proof} By using formula (\ref{exact}), one can write   $\int_{Q_\epsilon} D^3 v_{\epsilon} : D^3(T_\epsilon \varphi) \, \diff{x}$ as an integral over $\widehat W_{\epsilon}\times Y\times (-1 , 0)$ for suitable combinations of derivatives of  $\hat v_{\epsilon }$ and $\hat T_\epsilon \varphi$. Then using Lemmas~\ref{lemma: unfolding convergence},~\ref{lemma: convergence b}, one can prove that the limit of $\int_{Q_\epsilon} D^3 v_{\epsilon} : D^3(T_\epsilon \varphi) \, \diff{x}$ as $\epsilon \to 0$ equals
\small{
\begin{multline}\label{strangeterm}
 - 3 \int_W \int_{Y \times (-1,0)} D^2_y\Big(\frac{\partial \hat{v}}{\partial y_N} \Big) : D^2_y (b(\bar{y})(1 + y_N)^4) \,\diff{y}\, \frac{\partial^2\varphi}{\partial x_N^2}(\bar{x}, 0) \diff{\bar{x}}\\
- \int_W \int_{Y \times (-1,0)} y_N (D^3_y(\hat{v}) : D^3 (b(\bar{y})(1+y_N)^4) \,\diff{y}\, \frac{\partial^2\varphi}{\partial x_N^2}(\bar{x}, 0) \diff{\bar{x}}.
\end{multline}}
\normalsize This combined with (\ref{eq: limit easy terms}) and  (\ref{eq: limit less easy terms}) allows to pass to the limit in  (\ref{eq: Poisson 1}) and obtain the validity of (\ref{weakmacro}).
\end{proof}

The term (\ref{strangeterm}) appearing in (\ref{weakmacro}) plays the role of the so-called strange term, typical of many homogenization problems. We plan to write it in a more explicit way in order to complete the proof of statement (ii) in Theorem~\ref{mainthm}.  To do so, we characterise $\hat v$ as the solution to a suitable boundary value problem by proceeding as follows.
Let $\psi \in C^{\infty}(\overline{W} \times \overline{Y} \times ]-\infty, 0])$ be such that $\supp{\psi} \subset C \times \overline{Y} \times [d,0]$ for some compact set $C\subset W$ and $d\in]-\infty, 0[$ and such that $\psi(\bar{x},\bar{y},0) = \nabla\psi(\bar{x},\bar{y},0) = 0$ for all $(\bar{x},\bar{y})\in W \times Y$. Assume also $\psi$ to be $Y$-periodic in the variable $\bar{y}$. We set
$
\psi_\epsilon(x) = \epsilon^{\frac{5}{2}} \psi \Big(\bar{x},\frac{\bar{x}}{\epsilon}, \frac{x_N}{\epsilon}\Big),
$
for all $\epsilon >0$, $x \in W \times ]-\infty, 0]$. Then, by the $Y$-periodicity and the vanishing conditions imposed on $\psi$, $T_\epsilon \psi_\epsilon \in W^{3,2}(\Omega_\epsilon)\cap W^{2,2}_0(\Omega_\epsilon)$ for sufficiently small $\epsilon >0$, hence we can use it in the weak formulation of the problem in $\Omega_\epsilon$, getting
\begin{equation} \label{twoscaleweak}
\int_{\Omega_\epsilon} D^3v_\epsilon : D^3T_\epsilon \psi_\epsilon\,\diff{x} + \int_{\Omega_\epsilon} v_\epsilon T_\epsilon \psi_\epsilon \,\diff{x} = \int_{\Omega_\epsilon} f_\epsilon T_\epsilon \psi_\epsilon \,\diff{x}.
\end{equation}
Passing to the limit in (\ref{twoscaleweak}) yields the limit problem for $\hat v$.

\begin{theorem}
\label{thm: conditions on v}
Let $\hat{v} \in L^2(W,w^{3,2}_{Per_Y}(Y\times(-\infty,0)))$ be the function from Theorem~\ref{thm: macroscopic limit}. Then
\begin{equation}\label{weakmicro}
\int_{W\times Y \times (-\infty,0)} D^3_y\hat{v}(\bar{x},y) : D^3_y\psi (\bar{x},y) \diff{\bar{x}} \diff{y} = 0,
\end{equation}
for all $\psi \in L^2(W,w^{3,2}_{Per_Y}(Y\times(\infty,0)))$ such that $\psi(\bar{x},\bar{y},0) = \nabla\psi(\bar{x},\bar{y},0)= 0$ on $W\times Y$. Moreover, for any $i,j=1,\dots, N-1$, we have
\begin{equation}
\label{eq: Casado1}
\frac{\partial^2\hat{v}}{\partial y_i \partial y_N}(\bar{x}, \bar{y},0) = \frac{\partial b}{\partial y_i}(\bar{y}) \frac{\partial^2 v}{\partial x^2_N}(\bar{x}, 0) \quad\quad \textup{on $W \times Y$},
\end{equation}
and
\begin{equation}
\label{eq: Casado2}
\frac{\partial^2\hat{v}}{\partial y_i \partial y_j}(\bar{x}, \bar{y},0) = 0 \quad\quad \textup{on $W \times Y$}.
\end{equation}

\end{theorem}

\begin{proof}
By approximation we can assume that $\psi$ is smooth, with support described as above. Then it is easy to see that $\int_{\Omega_\epsilon} v_\epsilon T_\epsilon \psi_\epsilon \,\diff{x}$,  $\int_{\Omega_\epsilon} f_\epsilon T_\epsilon \psi_\epsilon \,\diff{x}$,  $\int_{\Omega_\epsilon\setminus \Omega } D^3v_\epsilon : D^3T_\epsilon \psi_\epsilon\,\diff{x}\to 0 $  as $\epsilon \to 0$.  Moreover, a slight modification of \cite[Lemma 8.47]{ArrLamb} combined with Lemma~\ref{lemma: unfolding convergence} yields
$
\int_\Omega D^3v_\epsilon : D^3T_\epsilon \psi_\epsilon\,\diff{x} \to \int_{W\times Y\times (-\infty,0)} D_y^3\hat{v}(\bar{x},y) : D_y^3\psi(\bar{x},y)\,\diff{\bar{x}} \diff{y}
$. Thus, passing to the limit in (\ref{twoscaleweak}) we obtain (\ref{weakmicro}).
 Differentiating the equality $\nabla v_\epsilon(\bar{x},g_\epsilon(\bar{x}))= 0$ which holds for all $\bar{x}\in W$, we get that for any $i,j=1, \dots , N-1$
\begin{equation}\label{idprep}
\frac{\partial^2 v_\epsilon}{\partial x_i \partial x_j} (\bar{x},g_\epsilon(\bar{x})) + \frac{\partial^2 v_\epsilon}{\partial x_i \partial x_N} (\bar{x},g_\epsilon(\bar{x})) \frac{\partial g_\epsilon(\bar{x})}{\partial x_j} = 0, \quad\quad \textup{for all $\bar{x}\in W$ }.
\end{equation}
Hence, setting
$
V_\epsilon^{ij} = \Big(0,\dots,0,-\frac{\partial^2 v_\epsilon}{\partial x_i \partial x_N}, 0,\dots,0, \frac{\partial^2 v_\epsilon}{\partial x_i \partial x_j}\Big),
$
for all $i=1,\dots, N$, $j=1,\dots,N-1$ we get
$
V_\epsilon^{ij} \cdot \nu_\epsilon = 0$, on $\Gamma_\epsilon$,
where $\Gamma_{\epsilon}$  is the part of the boundary of $\Omega_{\epsilon}$ given by the graph of $g_{\epsilon}$ and $\nu_\epsilon$ is the unit outer normal to $\Gamma_\epsilon$.  We note that by Lemma~\ref{lemma: unfolding convergence}
\[
\frac{1}{\sqrt{\epsilon}} \frac{\partial}{\partial y_k} \Bigg(\frac{\widehat{\partial^2 v_\epsilon}}{\partial x_i \partial x_j} \Bigg) \rightharpoonup    \frac{\partial^3 \hat{v}}{\partial y_i \partial y_j \partial y_k}
\]
in $L^2(W \times Y \times ]-\infty, 0[)$ as $\epsilon \to 0$. Applying \cite[Lemma 4.3]{casado10}, we get that
$
\frac{\partial^2\hat{v}}{\partial y_i \partial y_j}(\bar{x}, \bar{y},0) = \frac{\partial b}{\partial y_j}(\bar{y}) \frac{\partial^2 v}{\partial x_N \partial x_i}(\bar{x}, 0)$ on $W \times Y$
for all $i=1,\dots,N$, $j=1,\dots,N-1$, and since $v \in W^{3,2}(\Omega) \cap W^{2,2}_0(\Omega)$ this implies the validity of  (\ref{eq: Casado1}) and (\ref{eq: Casado2}).
\end{proof}

The two scale problem (\ref{weakmicro}) can be written in a more explicit way by separation of variables. We need the following lemma.
\begin{lemma}

\label{lemma: V}
There exists $V\in w^{3,2}_{Per_Y}(Y \times (-\infty,0))$ satisfying the equation
$$
\int_{Y \times (-\infty,0)} D^3 V : D^3 \psi \,\diff{y} = 0,
$$
for all the test-functions $\psi \in w^{3,2}_{Per_Y}(Y \times (-\infty,0))$ such that $\psi(\bar{y},0) = 0 = \nabla \psi(\bar{y},0)$ on $Y$, and satisfying the boundary conditions
\[
\begin{cases}
V(\bar{y},0) = 0, &\textup{on $Y$},\\
\frac{\partial V}{\partial y_N}(\bar{y},0) = b(\bar{y}), &\textup{on $Y$}.
\end{cases}
\]
Function $V$ is unique up to a sum of a monomial in $y_N$ of the form $a y_N^2$. Moreover,  $V \in W^{6,2}_{Per_Y}(Y \times (d,0))$ for any $d < 0$ and
it satisfies the equation $
\Delta^3 V = 0$, in $Y \times (d,0) $
subject to the boundary condition
$
\frac{\partial^3 V}{\partial y_N^3} (\bar{y},0) = 0$ on $Y $. Finally,
\begin{equation}\label{idenergy}
\int_{Y \times (-1,0)} 3 D^2_y\Big(\frac{\partial V}{\partial y_N} \Big) : D^2_y (b(\bar{y})(1 + y_N)^4)
+ y_N (D^3_y V : D^3 (b(\bar{y})(1+y_N)^4) \,\diff{y} =K.
\end{equation}
where $K$ is as in Theorem~\ref{mainthm}.
\end{lemma}

\begin{proof}
The first part of the lemma can be proved by  standard direct methods of the calculus of variations and  regularity theory, as in \cite{ArrLamb}.  We now prove (\ref{idenergy}). Let $\phi$ be the real-valued function defined on $Y \times ]-\infty,0]$ by
$
\phi(y) =
y_N b(\bar{y}) (1+y_N)^4$ if $-1\leq y_N \leq 0$  and
$\phi(y) = 0$ if $ y_N < -1$.
Then $\phi \in W^{3,2}(Y \times (-\infty,0))$, $\phi(\bar{y}, 0) = 0$ and
$\nabla\phi(\bar{y},0) = (0,0,\dots,0, b(\bar{y}))$.
Note that the function $\psi = V - \phi$ is a suitable test-function in equation $\int_{Y \times (-\infty,0)} D^3 V : D^3 \psi \,\diff{y} = 0$. By plugging it in we get
\begin{multline}
\int_{Y\times (-\infty,0)} |D^3 V|^2\,\diff{y} = \int_{Y \times(-1,0)} D^3 V : D^3\phi\, \diff{y}\\
= 3\int_{Y \times (-1,0)} D^2_y\Big(\frac{\partial V}{\partial y_N} \Big) : D^2_y (b(\bar{y})(1 + y_N)^4) \diff{y}\\
+ \int_{Y \times (-1,0)} y_N (D^3_y V : D^3 (b(\bar{y})(1+y_N)^4) \,\diff{y}.
\end{multline}
Thus the left-hand side of (\ref{idenergy}) equals $  \int_{Y\times (-\infty,0)} |D^3 V|^2\, d{y}$. The second equality in (\ref{kappa}) can be proved by integrating repeatedly by parts.
\end{proof}

\begin{theorem}\label{last}
Let $V$ be the function defined in Lemma~\ref{lemma: V}. Let $v,\hat{v}$ be as in Theorem~\ref{thm: macroscopic limit}. Then up to the sum of monomials of the type $a(\bar{x})y_N^2$, we have that
\begin{equation}\label{modulo}
\hat{v}(\bar{x},y) = V(y) \frac{\partial^2 v}{\partial x_N^2}(\bar{x},0).
\end{equation}
In particular,  the strange term in  \eqref{strangeterm} equals
$
-K \int_W  \frac{\partial^2 v}{\partial x_N^2}(\bar{x},0) \frac{\partial^2 \varphi}{\partial x_N^2}(\bar{x},0) \diff{\bar{x}}
$
where $K$ is as in  (\ref{mainthm}).
\end{theorem}

\begin{proof}
The function $
\hat{v}(\bar{x},y) = V(y) \frac{\partial^2 v}{\partial x_N^2}(\bar{x},0)$
satisfies problem (\ref{weakmicro}) subject to the boundary conditions (\ref{eq: Casado1}), (\ref{eq: Casado2}). Since the solution to such problem is unique up to the sum of monomials  of order $2$ in $y_N$, we conclude that the first part of the statement holds. The second part of the statement follows by  replacing $\hat{v}(\bar{x},y)$ by $V(y) \frac{\partial^2 v}{\partial x_N^2}(\bar{x},0)$ in \eqref{strangeterm} and using (\ref{idenergy}).
\end{proof}
It is now clear that combining Theorem~\ref{thm: macroscopic limit}  with Theorem~\ref{last} provides the proof of statement (ii) of Theorem~\ref{mainthm}.

{\bf Case $\alpha <3/2$.} In this case, it is not necessary to use the operator $T_{\epsilon }$ defined above, because it turns out that the limit energy space is not $W^{3,2}(\Omega )\cap W^{2,2}_0(\Omega)$ but  $W^{3,2}(\Omega )\cap W^{2,2}_0(\Omega)\cap W^{3,2}_{0,W}(\Omega)$ where $W^{3,2}_{0,W}(\Omega )$ is defined as the closure in $W^{3,2}(\Omega )$ of the $C^{\infty }$-functions which vanish in a neighbourhood of $W$. (Note that $W^{3,2}_{0,W}(\Omega )$ can be equivalently defined as the space of those functions  $u\in W^{3,2}(\Omega )$ such $D^{\alpha }u=0$ on $W$ for all $|\alpha |\le 2$.) In fact, since $\alpha <3/2$ and  the solution $v_{\epsilon }$ of problem (\ref{eq: Poisson prblm}) satisfies (\ref{idprep}), by \cite{casado10} the vector fields $V^{ij}_{\epsilon}$ defined above converge weakly in $W^{1,2}_{0,W}(\Omega )$. This implies by \cite[Lemma 4.3, Theorem 5.1]{casado10} that $D^{\alpha }v=0$ on $W$ for all $|\alpha |=2$, which  gives
$v\in W^{3,2}_{0,W}(\Omega )$.  Let $\varphi \in W^{3,2}(\Omega) \cap W^{2,2}_0(\Omega)\cap W^{3,2}_{0,W}(\Omega)$ be  fixed. Let $\varphi_0$ be the function obtained by extending  $\varphi$ by zero outside $\Omega$. It is straightforward to prove that $\varphi_0\in W^{3,2}(\Omega_{\epsilon}) \cap W^{2,2}_0(\Omega_{\epsilon})$ hence it is possible to test $\varphi_0$ in the weak formulation of problem (\ref{eq: Poisson prblm}) to obtain
$
\int_{\Omega_\epsilon} D^3v_\epsilon : D^3 \varphi_0 \, \diff{x} + \int_{\Omega_\epsilon}v_\epsilon  \varphi_0\, \diff{x} = \int_{\Omega_\epsilon} f_\epsilon  \varphi_0\, \diff{x}
$. By passing to the limit in this equality as $\epsilon\to 0$, one easily obtain that $
\int_{\Omega} D^3v : D^3 \varphi \, \diff{x} + \int_{\Omega}v  \varphi \, \diff{x} = \int_{\Omega } f \varphi \, \diff{x}
$ which concludes the proof of statement (iii).

\vspace{6pt}

{\bf Acknowledgments.} The first author is partially supported by grants MTM2012-31298 MINECO, Spain and Ayuda UCM-BSCH a Grupos de Investigaci\'on: Grupo de Investigaci\'on CADEDIF--920894. The second and third author  acknowledge financial support from the research project `Singular perturbation problems for differential operators' Progetto di Ateneo of the University of Padova and from the research project `INdAM GNAMPA Project 2015 - Un approccio funzionale analitico per problemi di perturbazione singolare e di omogeneizzazione'. The second and the third author are also members of the Gruppo Nazionale per l'Analisi Matematica, la Probabilit\`{a} e le loro Applicazioni (GNAMPA) of the Istituto Nazionale di Alta Matematica (INdAM).

\end{document}